\documentclass[12pt,a4paper,oneside]{amsart}

\usepackage{amsfonts, amsmath, amssymb, amsthm, amscd}
\usepackage{hyperref}
\hypersetup{
  colorlinks   = true, 
  urlcolor     = blue, 
  linkcolor    = blue, 
  citecolor   = red 
}
\usepackage{anysize}
\marginsize{2cm}{2cm}{1.5cm}{1.5cm}
\usepackage[pdftex]{graphicx}

\newtheorem{theorem}{Theorem}[section]
\newtheorem{lemma}[theorem]{Lemma}

\newtheorem{proposition}[theorem]{Proposition}

\newtheorem{corollary}[theorem]{Corollary}

\theoremstyle{definition}
\newtheorem{definition}[theorem]{Definition}

\theoremstyle{remark}
\newtheorem{remark}[theorem]{Remark}

\numberwithin{equation}{section}

\newcommand{\id}{\mathop{\mathrm{id}}}
\newcommand{\CAT}{\mathop{\mathrm{CAT}}}

\newcommand{\Rho}{\mathrm{P}}
\newcommand{\umink}{\mathop{\overline{\mathcal M}}\nolimits}
\newcommand{\lmink}{\mathop{\underline{\mathcal M}}\nolimits}
\newcommand{\ugauss}{\mathop{\overline{\mathcal G}}\nolimits}
\newcommand{\lgauss}{\mathop{\underline{\mathcal G}}\nolimits}
\newcommand{\gauss}{\mathop{{\mathcal G}}\nolimits}

\DeclareMathOperator{\vol}{vol}

\DeclareMathOperator{\dist}{dist}

\renewcommand{\epsilon}{\varepsilon}

\renewcommand{\phi}{\varphi}
\renewcommand{\kappa}{\varkappa}

\title{Waist of balls in hyperbolic and spherical spaces}

\author{Arseniy~Akopyan{$^\spadesuit$}}

\email{akopjan@gmail.com}

\author{Roman~Karasev{$^\clubsuit$}}

\email{r\_n\_karasev@mail.ru}
\urladdr{http://www.rkarasev.ru/en/}

\address{{$^\spadesuit$} Institute of Science and Technology Austria (IST Austria), Am Campus 1, 3400 Klosterneuburg, Austria}
\address{{$^\clubsuit$} Moscow Institute of Physics and Technology, Institutskiy per. 9, Dolgoprudny, Russia 141700}
\address{{$^\clubsuit$} Institute for Information Transmission Problems RAS, Bolshoy Karetny per. 19, Moscow, Russia 127994}

\thanks{{$^\clubsuit$} Supported by the Russian Foundation for Basic Research grant 18-01-00036}

\subjclass[2010]{49Q20,53C20,53C23}

\begin{document}
	
\begin{abstract}
In this paper we find a tight estimate for Gromov's waist of the balls in spaces of constant curvature, deduce the estimates for the balls in Riemannian manifolds with upper bounds on the curvature ($\mathrm{CAT}(\kappa)$-spaces), and establish similar result for normed spaces.
\end{abstract}

\maketitle

\section{Introduction}

The Gromov--Memarian waist of the sphere theorem~\cite{grom2003,mem2009,karvol2013} asserts that given a continuous map \mbox{$f : \mathbb{S}^n \to Y$}, with $\mathbb S^{n}$ the unit round sphere, $Y$ an $(n-k)$-dimensional manifold, and map having zero degree if $k=0$ (in what follows we only consider the case $k>0$ and do not care about the degree), it is possible to find $y\in Y$ such that for every $t\ge 0$
\[
\vol \nu_t(f^{-1}(y), \mathbb S^n) \ge \vol \nu_t(\mathbb S^k, \mathbb S^n),
\]
where $\nu_t(X, M)$ denotes the $t$-neighborhood of $X$ in the Riemannian manifold $M$.

Going to the limit $t \to 0$ this proves the result of Almgren (usually referred to \cite{almgren1965theory}, but this text is not available to the authors), that for any sufficiently regular map $f:\mathbb{S}^n \to \mathbb R^{n-k}$ with all fibers piece-wise smooth manifolds, there exists $y \in \mathbb R^{n-k}$ such that the $k$-dimensional volume of the fiber $f^{-1}(y)$ is greater or equal to the volume of the sphere $\mathbb{S}^{k}$. For merely continuous maps, the Gromov--Memarian theorem proves that $\lmink_k f^{-1}(y) \ge \vol_k \mathbb S^k$, where $\lmink_k$ denotes \emph{the lower Minkowski content}, which is defined in the following way: Let $M$ be a Riemannian manifold of dimension $n$ and $X\subseteq M$ be its subset, 
\begin{equation}
\label{equation:lower-mink}
\lmink_k (X, M) := \liminf_{t\to+0} \frac{\vol \nu_t(X, M)}{v_{n-k}t^{n-k}}\quad\text{and}\quad \umink_k (X, M) := \limsup_{t\to+0} \frac{\vol \nu_t(X, M)}{v_{n-k}t^{n-k}},
\end{equation}
where $v_m$ is the volume of the $m$-dimensional Euclidean unit ball. These values are called \emph{lower and upper Minkowski $k$-dimensional content} and are normalized to coincide with the Riemannian $k$-dimensional volume of $X$ in the case $X$ is a smooth submanifold of $M$. We write simply $\lmink_k (X)$ instead of $\lmink_k (X, M)$ when the ambient manifold is assumed, but in Section~\ref{sec:monotonicity} the dependence on the ambient manifold will play an essential role.

Gromov defines \emph{$k$-waist} of a set $X$ as the infimum of the numbers of $w>0$ for which $X$ admits a continuous map $X \rightarrow \mathbb{R}^{n-k}$, such that the $k$-dimensional volume of the preimage of any point $y\in \mathbb{R}^{n-k}$ is not greater than $w$. So, the $k$-waist of the unit sphere $\mathbb{S}^n$ in terms of the lower Minkowski content of fibers equals the volume of $\mathbb{S}^k$. In fact, the question of how we define ``$k$-volume'' is rather subtle. One can consider $k$-dimensional Riemannian volume and ask for sufficiently regular maps, consider the lower Minkowski content of fibers for arbitrary continuous maps, or even consider the Hausdorff measure of the fibers; the latter case seems to be the hardest, see the discussion and the references in~\cite{ahk2016}.

Not much was known about waists of other Riemannian manifolds. But recently Klartag in~\cite{klartag2016} proved that the $k$-waist of the unit cube equals $1$. This generalizes the Vaaler theorem \cite{vaaler1979} stating that the volume of any section of the unit cube by a $k$-plane passing through its origin has $k$-volume at least $1$. Klartag's idea was to transport the Gaussian measure to the uniform measure in the cube by a $1$-Lipschitz map and apply Gromov's waist theorem for the Gaussian measure \cite{grom2003}, see Section~\ref{section:densities} for the statement. In addition Klartag solved Guth's problem on waists of parallelepipeds in $\mathbb{R}^n$ 
and presented some general results for the waists of convex bodies that we also consider in Section~\ref{section:norm}.

In \cite{ak2016ball} the current authors proved that the $k$-waist of the unit ball (in the sense of the lower Minkowski content) in $\mathbb{R}^n$ equals the volume of the $k$-dimension unit ball. That proof is based on an application of the Archimedes lemma about projecting the uniform measure of $\mathbb S^{n+1}$ to a uniform measure in the Euclidean ball $B^n$. Remark \ref{remark:no-neighborhood-version} shows that in fact it is impossible to make a tight $t$-neighborhood version of the waist theorem for the Euclidean ball, thus justifying the passage to the lower Minkowski content version.

Other estimates of waists in different situations were obtained in \cite{ahk2016}, including more generalizations of Vaaler's theorem, discussion of the waist in terms of ``families of cycles'', and some results on waists in terms of the Hausdorff measure of the fiber. In particular, it was proved that the $1$-waist (in terms of the Hausdorff measure) of a convex body in $\mathbb{R}^n$ coincides with the width of the body.

In the review \cite[2.6]{gromov2014number} Gromov rises the question of finding the waist of balls in symmetric spaces; we answer this question for the cases of the sphere and the hyperbolic space:

\begin{theorem}
\label{theorem:model-ball-waist}
\emph{(The spherical case)}
Let $B(R)\subset \mathbb{S}^n$ be a ball of radius $R$, $R<\pi$, $Y$ an $(n-k)$ dimensional manifold, and $k>0$.
Then for any continuous map $f : B(R) \to Y$ it is possible to find $y\in Y$ such that $\lmink_k f^{-1}(y)$ is at least the volume of the $k$-dimensional ball of radius $R$ in $\mathbb{S}^{k}$.

\emph{(The hyperbolic case)}
Let $B(R)\subset \mathbb{H}^n$ be a ball of radius $R$, $Y$ an $(n-k)$ dimensional manifold, and $k>0$.
Then for any continuous map $f : B(R) \to Y$ it is possible to find $y\in Y$ such that $\lmink_k f^{-1}(y)$ is at least the volume of the $k$-dimensional ball of radius $R$ in $\mathbb{H}^{k}$.
\end{theorem}

The idea of the proof is to radially shrink the measure coming from the stereographic projection of the sphere and apply the Gromov--Memarian waist theorem. In fact, this idea was outlined in~\cite[(e) on pages 477--478]{grom2010} for the Euclidean and the hyperbolic case for the waist of sufficiently regular maps. In this paper we show that it also works for the spherical case and extends to continuous maps with the Minkowski content of the fibers. 

\subsection{Results on $\CAT(\kappa)$ spaces}

As a corollary we obtain some results for metric balls in $\CAT(\kappa)$ Riemannian manifolds. Let us remind that a Riemannian manifold is a $\CAT(\kappa)$-space (see~\cite{grom2001}) if it is complete, simply connected, and has sectional curvature $\le \kappa$ everywhere, see~\cite[Page 118, (a')]{grom2001}.

Denote by $\mathbb M^n_{\kappa}$ the $n$-dimensional hyperbolic space with curvature $\kappa$ in case $\kappa<0$, Euclidean space for $\kappa=0$, and the sphere of radius $1/\sqrt{\kappa}$, for $\kappa>0$.

We will use the following \emph{triangle comparison} property of $\CAT(\kappa)$ spaces (it is the definition of $\CAT(\kappa)$ in~\cite[Page 117--118]{grom2001}): Any triangle $ABC$, with perimeter less than $2\pi/\sqrt{\kappa}$ in case $\kappa>0$, can be compared to a triangle $A'B'C'\subset \mathbb M^2_\kappa$ in the following way. If we have $\dist(A,B) = \dist(A',B')$, $\dist(A,C) = \dist(A',C')$ and the angle between the shortest segments $[A,B]$ and $[A,C]$ is the same as the angle between $[A',B']$ and $[A',C']$ then $\dist(B,C) \ge \dist(B',C')$.
It follows easily from the angle comparison property (see \cite[Theorem 3.9.1.]{toponogov2006differential}) and the monotone dependence of the length of $[B'C']$ on the angle $A'$.
Now we state the results for $\CAT(\kappa)$ spaces:

\begin{theorem}
\label{theorem:CAT-waist}
Let $M$ be an $n$-dimensional complete Riemannian manifold, which is $\CAT(\kappa)$; and let $B_M(R)\subset M$ be a ball of radius $R$ there, $R<\pi/\sqrt{\kappa}$ in case $\kappa>0$. Then for any continuous map $f : B_M(R) \to Y$, with $Y$ an $(n-k)$-dimensional manifold and $k>0$, it is possible to find $y\in Y$ such that $\umink_k f^{-1}(y)$ is at least the volume of the $k$-dimensional ball in the model space~$\mathbb M_{\kappa}^k$.
\end{theorem}

We also state the following consequence of the above theorem for the whole $\CAT(1)$ manifold:

\begin{corollary}
\label{corollary:cat1waist}
Let $M$ be an $n$-dimensional complete Riemannian manifold, which is $\CAT(1)$. Then for any continuous map $f : M \to Y$, with $Y$ an $(n-k)$ dimensional manifold and $k>0$, it is possible to find $y\in Y$ such that 
\[
\umink_k f^{-1}(y) \ge \vol_k \mathbb S^k.
\]
\end{corollary}

We conjecture that the above two results, proved for the upper Minkowski content, hold for the lower Minkowski content as well, but could not handle the technicalities to establish this. 

\subsection{Outline of the paper}

The paper is organized in the following way: In Section~\ref{section:densities} we remind the definition of the Minkowski content for metric spaces with a density. In Section~\ref{section:radial} we show that certain type of radial transformations of a density in $\mathbb{R}^n$ preserves the waist inequality. After that, in Section~\ref{section:model-riemannian-balls} we apply this transformation to the densities in the Euclidean space obtained by the conformal projection from the model spaces. This proves Theorem~\ref{theorem:model-ball-waist} for sufficiently regular maps and the waist in terms of the Riemannian $k$-volume. In the end of this section we give proofs of Theorem \ref{theorem:CAT-waist} and Corollary \ref{corollary:cat1waist} for sufficiently regular maps.

Section~\ref{sec:mink-waist} addresses the waist in terms of the lower Minkowski content by modifying some parts of the Gromov--Memarian argument. In \ref{sec:gromov-discontinius} we show how to change the proof of the Gromov--Memarian theorem so that it remains valid for certain discontinuous maps from the sphere. In \ref{ssec:minkowski-waist} we give the proof of Theorem~\ref{theorem:model-ball-waist} in its full form. In Section~\ref{section:norm} we give a version of the waist theorem for normed spaces, similar to a result by Klartag.

In the Section \ref{sec:monotonicity} we discuss the monotonicity of the Minkowski content under $1$-Lipschitz maps, prove Theorem~\ref{theorem:CAT-waist}, and define a Gaussian version of the Minkowski content that behaves better in terms of monotonicity and taking Cartesian products.


\subsection{Acknowledgments.}
The authors thank Anton Petrunin for pointing out the properties of $\CAT(\kappa)$ spaces that we use and the unknown referees for numerous useful remarks and corrections.

\section{Some observations and notation}
\label{section:densities}

If we consider the problem for sufficiently regular smooth maps and measure the fiber with the Riemannian $k$-volume then the general idea becomes very clear. Suppose that $X_1$ and $X_2$ are equipped with a Riemannian metric and a smooth one-to-one map $h:X_2 \rightarrow X_1$ does not decrease $k$-volumes, this assumption has an expression in terms of the derivative of the map. Then the waist of $X_2$ is no less than the waist of $X_1$. Indeed, for any map $f:X_2 \rightarrow Y$, there is a map $h^{-1}\circ f:X_1 \rightarrow Y$, for which there is a $y\in Y$ with big $k$-dimensional volume of $h^{-1}\circ f(y)$. Since $h$ does not decrease $k$-volumes, we obtain that $f^{-1}(y)$ has also a big $k$-volume.

The map constructed in the further sections can be considered as a map for spherical and hyperbolical balls as $X_2$ and the sphere without a point as $X_1$, for which we modify the Gromov--Memarian theorem. For the proof of Theorem~\ref{theorem:CAT-waist} we use balls in model spaces as $X_1$ and connect $X_1$ and $X_2$ through exponential maps.

In the middle of our argument there also appear Riemannian volumes with densities, so we make some general remarks first. In~\cite[Question 3.1]{grom2003} Gromov asked about the assumptions on a rotation invariant density in $\mathbb{R}^n$ (and a rotation invariant Riemannian metric) that guarantee a $k$-waist inequality with the minimum attained at the $k$-plane through the origin. We consider $\mathbb R^n$ with the standard Euclidean metric and a rotation invariant density and give examples where the waist defined in terms of the weighted Riemannian $k$-volume is indeed attained at any linear $k$-subspace through the origin.

Now let us choose some notation. We consider a density $\rho$ on an $n$-dimensional Riemannian manifold $M$ and define the \emph{$k$-dimensional weighted Riemannian volume} of a $k$-dimensional (piece-wise) smooth submanifold $X\subset M$ as
\[
\vol_{k, \rho} X = \int_X \rho \vol_{g|_X},
\]
where $\vol_{g|_X}$ is the Riemannian $k$-dimensional density on $X$ corresponding to the restriction of $g$ to $X$, given in local coordinates $u_1,\ldots, u_k$ as $\sqrt{\det g|_X} du_1\dots du_k$. 

We will frequently use the fact that the Riemannian volume for nice sets (like piece-wise smooth manifolds) can be calculated as the Minkowski content, so that for its weighted version we have the formula:
\begin{equation}
\label{eq:mink-content}
\vol_{k, \rho} = \lim_{t\to +0} \frac{\int_{\nu_t (X, M)} \rho \vol_g}{v_{n-k}t^{n-k}},
\end{equation}
where $\nu_t$ denotes the $t$-neighborhood of $X$ in $M$, $\vol_g$ is the Riemannian density in $M$ corresponding to the metric $g$, and $v_\ell$ is the volume of the unit $\ell$-dimensional Euclidean ball.
We will use the notation $\vol_{n,\rho} = \rho \vol_g$ to simplify the formulas and will simply write for $n$-dimensional subsets $U\subseteq M$, $\vol_{n,\rho} (U) = \int_U \rho \vol_g$.
	
In~\cite{grom2003} Gromov proved the waist theorem for the Gaussian measure: Given a continuous map $f : \mathbb{R}^n \to Y$, where $\mathbb R^{n}$ is equipped with a Gaussian density $\rho = A e^{-\alpha |x|^2}$ ($\alpha, A > 0$) and $Y$ is an $(n-k)$-dimensional manifold, it is possible to find $y\in Y$ such that 
\[
\vol_{n,\rho} \nu_t(f^{-1}(y), \mathbb R^n) \ge \vol_{n,\rho} \nu_t(\mathbb R^k, \mathbb R^n).
\]
Speaking in terms of the weighted (lower) Minkowski content, we may conclude that 
\[
\lmink_{k,\rho} f^{-1}(y) \ge \lmink_{k,\rho} \mathbb R^k,
\]
if we put for a Riemannian manifold $M$ with a density $\rho$
\[
\lmink_{k,\rho} (X, M) = \liminf_{t\to+0} \frac{\vol_{n, \rho} \nu_t(X, M)}{v_{n-k}t^{n-k}}\quad\text{and}\quad \umink_{k,\rho} (X, M) = \limsup_{t\to+0} \frac{\vol_{n, \rho} \nu_t(X, M)}{v_{n-k}t^{n-k}}.
\]


\section{Behaviour of waists with weights under radial maps}	
\label{section:radial}

In this section we consider the Euclidean space with a density. By the $k$-waist of a density $\rho$ we call the infimum of the numbers of $w>0$ such that there is a regular (piece-wise real-analytic) map to a smooth manifold $\mathbb{R}^n \rightarrow Y^{n-k}$, such that the $k$-dimensional $\rho$-weighted volume of the fibers $f^{-1}(y)$ for all $y\in Y^{n-k}$ is not greater than $w$. The assumption that the map is piece-wise real analytic allows us to think of the fibers $f^{-1}(y)$ as piece-wise real-analytic submanifolds of certain dimension.

Suppose we are given a radial (only depending on the distance to the origin) density $\rho$, for which we know that its $k$-waist equals to $\vol_{k, \rho}(\mathbb{R}^k)$, here $\mathbb R^k\subset \mathbb R^n$ is a linear subspace. Our goal is to infer the same statement for another appropriately chosen radial density $\sigma$, connected to $\rho$ in the way we describe below.

Suppose $\phi:\mathbb{R}_{+} \rightarrow \mathbb{R}_{+}$ is an increasing continuous function, with $\phi(0)=0$ and having property:
\begin{equation}
	\label{eq:condition on function f}
	x \phi'(x)>\phi(x).
\end{equation}
This property has a simple geometric meaning: Any tangent line to the graph of $\phi$ passes below the origin.

Now, consider the map $F:\mathbb{R}^n \rightarrow \mathbb{R}^n$ that maps any point $\textbf{x}\in \mathbb{R}^n$ to $\phi(|\textbf{x}|)\frac{\textbf{x}}{|\textbf{x}|}$, which we assume to be extended continuously so that $F(0) = 0$. The following formulas are written for $\textbf{x}\neq 0$, this does not affect the argument since we are mostly integrating something.

The derivative of $F$ has eigenvalue $\phi'(|\mathbf{x}|)$ in the radial direction $\frac{\textbf{x}}{|\textbf{x}|}$; and eigenvalues $\frac{\phi(|\mathbf{x}|)}{|\mathbf{x}|}$ ($d-1$ times) in the directions orthogonal to the radial direction. It is easy to see that \eqref{eq:condition on function f} is equivalent to saying that the largest eigenvalue of the derivative $DF$ is $\phi'(|\mathbf{x}|)$ in the radial direction. From the description of the eigenvalues of $DF$ we have that the restriction of $D_{\mathbf{x}}F$ to any $k$-dimensional linear subspace (with its inherent Euclidean structure) has determinant at most $\phi'(|\mathbf{x}|) \left(\frac{\phi(|\mathbf{x}|)}{|\mathbf{x}|}\right)^{k-1}$ by its absolute value and hence for any $k$-dimensional submanifold $X$ we have:
\begin{multline}
\label{eq:k-measure-transform}
\int_{F(X)} \rho(\mathbf y)d\vol_k(\mathbf y) = \int_X \rho(F(\mathbf x)) \left| \det D_{\mathbf{x}}F|_X\right| d\vol_k(\mathbf x) \le\\
\le  \int_X \phi'(|\mathbf{x}|)\left(\frac{\phi(|\mathbf{x}|)}{|\mathbf{x}|} \right)^{k-1} \rho(F(\mathbf x))d\vol_k(\mathbf x).
\end{multline}

Suppose we are given a density $\sigma$ and a function $\phi$ such that
\begin{equation}
	\label{eq:definition of mu}
	 \sigma(\mathbf{x}) = \phi'(|\mathbf{x}|) \left(\frac{\phi(|\mathbf{x}|)}{|\mathbf{x}|}\right)^{k-1}  \rho(F(\mathbf{x})).
\end{equation}
Note that $\vol_{k,\sigma}(\mathbb{R}^k)=\vol_{k,\rho}(\mathbb{R}^k)$. Indeed, 
\begin{multline*}
	\vol_{k, \sigma}(\mathbb{R}^k)=
	\int_{\mathbb{R}^k} \sigma(\mathbf{x}) d\vol_k(\mathbf x)=
	\int_{\mathbb{R}^k} 
		\rho(F(\mathbf{x})) \phi'(|\mathbf{x}|) \left(\frac{\phi(|\mathbf{x}|)}{|\mathbf{x}|}\right)^{k-1} 
	d\vol_k(\mathbf x)=\\
	=
	\int_{\mathbb{R}^k} \rho(\mathbf y) d\vol_k(\mathbf y)=
	\vol_{k,\rho} (\mathbb{R}^k)
\end{multline*}

Assume the contrary to what we want to prove: The $k$-waist of $\sigma$ is less than $\vol_{k,\sigma} (\mathbb{R}^k)=\vol_{k,\rho}(\mathbb{R}^k)$. This means the existence of a regular map $f:\mathbb{R}^n \rightarrow Y^{n-k}$ for which the preimage of any point has $\sigma$-weighted $k$-volume smaller than $\vol_{k,\sigma}(\mathbb{R}^k)-\varepsilon$ for some $\varepsilon>0$. Consider the map $s=f\circ F^{-1}$. By the assumption on $\rho$, for some $y\in Y$, the measure of $s^{-1}(y)$ is at least $\vol_{k\,\rho}(\mathbb{R}^k)$. Taking into account \eqref{eq:k-measure-transform}, we obtain

\begin{multline*}
\vol_{k,\rho} (\mathbb{R}^k)\leq
\int_{s^{-1}(y)}\rho(\mathbf{y})d\vol_k(\mathbf y)
=
\int_{f^{-1}(y)}
 \rho(F(\mathbf{x})) \left|\det D_{\mathbf{x}}F|_{f^{-1}(y)}\right| d\vol_k(\mathbf x)
\leq\\
\leq \int_{f^{-1}(y)}
\left(\frac{1}{\phi'(|\mathbf{x}|)}\left(\frac{|\mathbf{x}|}{\phi(|\mathbf{x}|)}\right)^{k-1}
\sigma(\mathbf{x})\right) \left(\phi'(|\mathbf{x}|) \left(\frac{\phi(|\mathbf{x}|)}{|\mathbf{x}|}\right)^{k-1}\right)  d\vol_k(\mathbf x)=\\
=\int_{f^{-1}(y)} \sigma({\mathbf{x}}) d\vol_k(\mathbf x)
\leq k \text{-waist of }\sigma.
\end{multline*}

We have a contradiction. Thus we have proved:

\begin{theorem}
\label{theorem:density-radial-map}
If the $k$-waist of a radial density $\rho$ equals $\vol_{k,\rho}(\mathbb R^k)$, a radial map $F$ is given by the function $\phi$ with $x \phi'(x)>\phi(x)$, and another density $\sigma$ satisfies 
\[
\sigma(\mathbf{x}) = \phi'(|\mathbf{x}|) \left(\frac{\phi(|\mathbf{x}|)}{|\mathbf{x}|}\right)^{k-1}  \rho(F(\mathbf{x})),
\]
then the density $\sigma$ also has $k$-waist equal to $\vol_{k,\sigma}(\mathbb R^k) = \vol_{k,\rho}(\mathbb R^k)$. In this theorem we speak about $k$-waist for sufficiently regular maps and fibers in terms of their Riemannian $k$-volume.
\end{theorem}

In the equality connecting $\sigma$ and $\rho$ it is convenient to replace the densities $\sigma$ and $\rho$ with 
\[
\sigma_k(|\mathbf x|) = \sigma(\mathbf x) |\mathbf x|^{k-1}\quad\text{and}\quad \rho_k(|\mathbf x|) = \rho(\mathbf x) |\mathbf x|^{k-1}
\]

Now \eqref{eq:definition of mu} becomes just
\begin{equation}
	\label{eq:definition of mu_k}
	\sigma_k(x)=\phi'(x)\rho_k(\phi(x)).
\end{equation}
In terms of $\psi=\phi^{-1}$ this becomes
\begin{equation}
	\label{eq:reversedefinition of mu_k}
	\psi'(x)\sigma_k(\psi(x))=\rho_k(x), 
\end{equation}
where $\psi$ will evidently satisfy the opposite assumption
\[
x \psi'(x)<\psi(x).
\]

\begin{remark}
\label{remark:k-ball-to-ball}
The geometric meaning of the integrated \eqref{eq:definition of mu_k} is that $F$ sends any $k$-dimensional Euclidean ball $B^k(r)\subset\mathbb R^n$ centered at the origin to another $k$-dimensional Euclidean ball $B^k(R)$ so that the $\sigma$-weighted $k$-volume of $B^k(r)$ equals the $\rho$-weighted $k$-volume of $B^k(R)$. This can be seen by comparing the derivatives by $r$ and by $R$ of the respective weighted volumes. 

The meaning of the inequality $x \phi'(x)>\phi(x)$ is that the local increase in the $k$-dimensional Riemannian $k$-volume under $F$ is maximal for $k$-submanifolds tangent to the radial direction.
\end{remark}

\section{Proof of the theorems for the $k$-Riemannian volume}
\label{section:model-riemannian-balls}

\subsection{Riemannian $k$-volume version of Theorem~\ref{theorem:model-ball-waist}}

Let us construct certain densities $\rho$ having $k$-dimensional waist equal to $\vol_{k,\rho}(\mathbb R^k)$. What we definitely know is the Gromov--Memarian theorem on the waist of the sphere~\cite{grom2003,mem2009}. Consider the sphere to be given by $|\mathbf x|^2 + y^2=1$ in $\mathbb R^{n+1}$ and project it stereographically from $(\mathbf 0, 1)$ on the plane $y=-1$. This map $S$ will be conformal with factor ${1+|\mathbf x|^2}$ (at the image point $\mathbf x$). Note that discontinuity at one point in the Gromov--Memarian theorem is allowed thanks to Theorems~\ref{theorem:disc-waist} and \ref{theorem:disc-waist-k1} below.

Therefore the $k$-dimensional Riemannian volume on the sphere is transformed to $\vol_{k,\rho}$ with 
\[
\rho(\mathbf x) = \frac{1}{(1+|\mathbf x|^2)^k},
\]
that is the Riemannian $k$-volume of a $k$-dimensional submanifold $X\subset\mathbb S^n$ equals the $\rho$-weighted Riemannian $k$-volume of $S(X)$. This allows us to conclude with

\begin{theorem}
For sufficiently regular maps and fibers in terms of their Riemannian $k$-volume, the $k$-waist of the radial density $\rho(x)=\frac{1}{(1+x^2)^k}$ equals $\vol_{k,\rho}(\mathbb R^k)$.
\end{theorem}

Note that this density depends on $k$ and does not depend on the ambient dimension $n$, so unlike the Gromov--Memarian result, or Gromov's theorem on the waist of the Gaussian density, we use a specific density for every particular value of $k$.

Now we want to start from the established case $\rho(x)=\frac{1}{(1+x^2)^k}$ and move to another density $\sigma(x)$. Denote by $\Sigma(t)$ and $\Rho(t)$ the functions $\int_0^t\sigma_k(x)dx$ and $\int_0^t \rho_k(x)dx$ respectively, recall that we put $\sigma_k(x) = \sigma(x)x^{k-1}$ and $\rho_k(x) = \rho(x)x^{k-1}$. Then we can rewrite \eqref{eq:reversedefinition of mu_k} in the form:
\[
\frac{d}{dx} \Sigma(\psi(x))=\Rho'(x)\ \forall x \Leftrightarrow  \Sigma(\psi(x))=\Rho(x)\ \forall x \Leftrightarrow \psi(x)=\Sigma^{-1}(\Rho(x)).
\]
 
The assumption $x \psi'(x)<\psi(x)$ is equivalent to that $\psi(x)/x$ monotonicaly decrease (just differentiate). Setting $x = \Rho^{-1}(y)$, there remains to prove that
\[
\frac{\psi(x)}{x} = \frac{\Sigma^{-1}(y)}{\Rho^{-1}(y)}
\] 
monotonically decrease.

Since $\psi$ is defined through the equality $\Sigma(\psi(x))= \Rho(x)$, for any $x_1<x_2$, we need to show that $\frac{\psi(x_1)}{x_1}> \frac{\psi(x_2)}{x_2}$. In view of the monotonicity of $\Sigma$, it is sufficient to show that  $\Sigma\left(\frac{\psi(x_1)}{x_1} x_2 \right)$ is greater than $\Rho(x_2)=\Sigma(\psi(x_2))$.
Denote $\frac{\psi(x_1)}{x_1}$ by $c$. 
Our aim is to show
\begin{equation}
	1=\frac{\int\limits_{0}^{cx_1} \sigma_k(t) dt}{\int\limits_{0}^{x_1} \rho_k(t) dt} <?
	\frac{\int\limits_{0}^{cx_2} \sigma_k(t) dt}{\int\limits_{0}^{x_2} \rho_k(t) dt}
\end{equation}
or equivalently
\begin{equation}
	\label{eq:comparison of S}
	\frac{1}{c} =\frac{\int\limits_{0}^{x_1} \sigma_k(ct) dt}{\int\limits_{0}^{x_1} \rho_k(t) dt} <?
	\frac{\int\limits_{0}^{x_2} \sigma_k(ct) dt}{\int\limits_{0}^{x_2} \rho_k(t) dt}
\end{equation}
``Gromov's inequality''~\cite[p.42]{cheeger-gromov-taylor1982} states that it is sufficient to prove that $\frac{\sigma_k(ct)}{ \rho_t(t)}$ is increasing to have the desired inequality.

Now assume we want to estimate the waist of a spherical cap (a ball in the spherical geometry). The whole sphere projects to give the density $\rho(x) = \frac{1}{(1+x^2)^k}$, while the cap will have density $\sigma(x) = \frac{A}{(1+x^2)^k}$ when $x\in [0, R']$ and $0$ for $x > R'$, where $R'=\frac{2}{\cos(R/2)}$ is a radius of the cap after the stereographic projection $S$, and $A>1$ is chosen so that the total integrals of $\rho_k = \frac{x^{k-1}}{(1+x^2)^k}$ and $\sigma_k(x) = \sigma(x)x^{k-1}$ coincide (and equals the volume of $\mathbb{S}^k$).
In this case we have to consider
\begin{equation}
	\label{eq: final spherical density equation}
	\frac{\sigma_k(ct)}{ \rho_k(t)}=
	\frac{
	\frac{A}{(1+(ct)^2)^k}c^{k-1}t^{k-1}
	}{
	\frac{1}{(1+t^2)^k}t^{k-1}
	} = A c^{k-1} \left(\frac{1+t^2}{1+(ct)^2}\right)^k
\end{equation}
and show it is increasing. A simple transformation gives
\begin{equation}
\frac{1+t^2}{1+c^2t^2}=c^{-2}+\frac{1-c^{-2}}{1+c^2t^2},
\end{equation}
which is clearly increasing if we show that $c<1$. 
Indeed, since $A>1$, the density for spherical cap is larger then the density coming from stereographic projection of the whole sphere, therefore for small $x$, $\psi(x)<x$, and as we have seen the ratio $\frac{\psi(x)}{x}$ only decrease in $x$.
This argument proves the spherical case of Theorem~\ref{theorem:model-ball-waist} for a regular map $f$. More precisely, we have proved the following:
%
\begin{proposition}
\label{proposition:cap-waist}
The waist of sufficiently regular maps $C \to Y^{n-k}$ for a spherical cap $C\subset \mathbb S^n$ and the Riemannian $k$-volume in $\mathbb S^n$ is attained at intersections $C\cap \mathbb S^k$ passing through the center of $C$.
\end{proposition}

Hyperbolic geometry in the Poincar\'e model has the conformal factor $\frac{1}{1-x^2}$. Thus repeating the argument for the spherical case in the right hand part of Equation~\eqref{eq: final spherical density equation} we obtain 
\[
Ac^{k-1}\left(\frac{1+t^2}{1-(ct)^2}\right)^k
\] 
(we are moving $\mathbb R^n$ to the unit ball and work in the range $ct < 1$), which obviously increases in $x$. This proves the hyperbolic case for regular maps. Let us formulate this:

\begin{proposition}
\label{proposition:hyp-waist}
The waist of sufficiently regular maps $B \to Y^{n-k}$ for a hyperbolic ball $B\subset \mathbb H^n$ and the Riemannian $k$-volume in $\mathbb H^n$ is attained at intersections $B\cap \mathbb H^k$ passing through the center of $B$.
\end{proposition}
%

\subsection{Riemannian $k$-volume version of Theorem \ref{theorem:CAT-waist} and Corollary \ref{corollary:cat1waist}}
\label{sec:cat-riemannian proofs}

Corollary~\ref{corollary:cat1waist} is obtained from Theorem~\ref{theorem:CAT-waist} by going to the limit $R\to\pi$, so we assume we only consider Theorem~\ref{theorem:CAT-waist} and work with metric balls. We use the term \emph{anti-$1$-Lipschitz} for maps between metric spaces that do not decrease the distance between pairs of points.

\begin{lemma}
\label{lemma:cat-anti-lipschitz}
If $M$ is an $n$-dimensional $\CAT(\kappa)$ Riemannian manifold and $\mathbb M_\kappa^n$ is the model space of constant curvature $\kappa$, then any metric ball $B_p(R)\subset M$ (assume $R<\pi/\sqrt{\kappa}$ for $\kappa>0$) is an anti-$1$-Lipschitz image of the ball $B'_q(R)\subset \mathbb M_\kappa^n$.
\end{lemma}

\begin{proof}
Since $M$ is $\CAT(\kappa)$, the exponential map $\exp_p : T_pM\to M$ can be considered a diffeomorphism between the balls of radius $R$ in the tangent space and in the manifold itself. Moreover, if we consider the similar exponential map on the model space $\exp_q : T_q \mathbb M^n_\kappa \to \mathbb M^n_\kappa$ then we obtain that the map
\[
g = \exp_p \circ\exp_q^{-1}
\]
is defined for all points of $B'_q(R)$ and does not decrease the distance from the triangle comparison property of $\CAT(\kappa)$.
\end{proof}

So we have an anti-$1$-Lipschitz map $h$ from $B'_q(R)\subset \mathbb M^n_\kappa$ to $B_p(R)\subset M$. In order to study the Riemannian $k$-volume of the fibers of $f$ we may observe that under the regularity assumptions on $f$, a fiber $X$ of $f\circ h$ has $k$-volume at least $\vol_k B^k(R)$, where $B^k(R)$ is the $k$-dimensional ball in the model space. After that we just observe that $h(X)$ is the corresponding fiber of $f$ and its Riemannian $k$-volume cannot decrease under an anti-$1$-Lipschitz map $h$.

\section{Proofs of the waist theorems for the lower Minkowski content}
\label{sec:mink-waist}

\subsection{Waist of the sphere for discontinuous maps}
\label{sec:gromov-discontinius}

Compared to the results of the previous section, Gromov and Memarian proved the waist of the sphere theorem in a more general setting, for arbitrary continuous maps and the sharp estimates for the volumes of all $t$-neighborhoods of the fiber $f^{-1}(y)$. We start with a generalization of Gromov's waist of the sphere theorem to the case when $f$ is not defined on the whole sphere, because in the previous section we already needed a particular case of this theorem with one point of discontinuity.

\begin{theorem}
\label{theorem:disc-waist}
Let $k>1$ and assume $V_1,\ldots, V_N$ are linear subspaces of $\mathbb R^{n+1}$ of dimension at most $k-1$ each. Given a continuous map 
\[
f : \mathbb S^n\setminus \bigcup_i V_i \to Y,
\]
with $\mathbb S^{n}$ the unit round sphere, $Y$ an $(n-k)$-dimensional manifold, it is possible to find $y\in Y$ such that for every $t>0$
\[
\vol \nu_t(f^{-1}(y), \mathbb S^n) \ge \vol \nu_t(\mathbb S^k, \mathbb S^n)
\]
and in particular
\[
\lmink_k f^{-1}(y) \ge \vol_k \mathbb S^k.
\]
\end{theorem}

\begin{proof}
In fact the proof in~\cite{grom2003,mem2009} works fine for this statement and we assume the familiarity of the reader with the structure of that proof, citing it mostly by the paper \cite{mem2009}.

That proof consisted in successively cutting the sphere $\mathbb S^n$ with hyperplanes to make a binary partition of the sphere. The normals of the hyperplanes of $i$th stage of the binary tree were chosen perpendicular to a certain prescribed subspace $L_i$ of dimension $k-1$, see \cite[Pages 10--13]{mem2009}. The subspaces $L_i$ were not quite arbitrary, but the assumption needed for the proof was that when $i\to \infty$ the subspaces $L_i$ must be dense in the Grassmannian $G_{k-1, n+1}$. This density assumption guarantees that the parts $P_i$ of the partition, when the number of parts tends to infinity, will be close to $(n-k)$-dimensional subspheres and are called \emph{pancakes} because of this.

Note that this assumption will be still satisfied if we choose the beginning of the list of subspaces $L_i$ in the argument containing our initially given list $V_1$, $\ldots$, $V_N$. Then the hyperplane cuts will contain all the $V_i$, and the open parts of the resulting partition will be contained in $\mathbb S^n\setminus \bigcup_i V_i$, where the map $f$ is continuously defined and where its values are used when applying the Borsuk--Ulam type argument of \cite[Proof of Theorem 4]{mem2009} to ensure that the ``centers'' of the pancakes go to the same point under $f$. The rest of the argument in the proof in~\cite{mem2009} works without change.

We also have a generalization to arbitrary manifold $Y$ similar to~\cite{karvol2013}. The cohomology triviality assumption from \cite[Lemma 3.7]{karvol2013} is satisfied this time because the configuration space \cite[the last paragraph]{karvol2013} is built from spheres of dimension $n-k+1$, perpendicular to the given $L_i$, that are mapped to $Y$ inducing zero map on the reduced cohomology. And the $f$-images of the centers of pancakes in \cite[Lemma 5.1]{karvol2013} need only be calculated on the domain of continuity of $f$.
\end{proof}

The previous theorem shows that the waist of the sphere theorem works in full strength for maps with, informally speaking, discontinuities of codimension $2$ compared to the fibers. In the case $k=1$ for the sphere with one point removed we only prove the following weaker result:

\begin{theorem}
\label{theorem:disc-waist-k1}
Let $x\in\mathbb S^n$ be a point in the unit round sphere. Given a continuous map
\[
f : \mathbb S^n\setminus \{x\} \to Y,
\]
where $Y$ is a smooth $(n-1)$-dimensional manifold, it is possible to find $y\in Y$ such that
\[
\lmink_1 f^{-1}(y) \ge \vol_1 \mathbb S^1 = 2\pi.
\]
\end{theorem}

\begin{proof}
This time we need to examine the proof in~\cite{grom2003,mem2009} in more detail. In this case ($k=1$ in our notation) we split the sphere by a binary partition into equal volume \emph{pancakes} $P_1,\ldots, P_N$, where $N$ is a big power of two, and apply the Borsuk--Ulam type argument of \cite[Proof of Theorem 4]{mem2009} to make the images of their ``centers'' equal. The parts are called \emph{pancakes} because they are $\delta$-close ($\delta\to +0$ as $N\to\infty$) to their respective $(n-1)$-dimensional geodesic subspheres $T_i$. The ``center'' of a pancake $P_i$ is chosen to be close to the point of maximum density of the projection of the volume of $P_i$ to $T_i$, see \cite[Section 5.1]{mem2009}.

Now choose $t>0$ that will be the size of the neighborhood of $f^{-1}(y)$, in this theorem we will only be interested in the limit behavior when $t\to +0$. Choose a much smaller $\delta>0$ (we will write $\delta \ll t$) and modify $f$ in the $\delta$-neighborhood of $x$, $\nu_\delta (x)$ so that the modified map becomes continuous. This is possible since $f$ can be continuously extended from the sphere $\partial\nu_\delta(x)$ to the ball $\mathbb S^n\setminus \nu_\delta(x)$ and therefore it is also possible to extend it continuously to the open ball $\nu_\delta(x)$. 

Apply the Borsuk--Ulam-type pancake argument from \cite[Proof of Theorem 4]{mem2009} to the modified map. It produces some $y_\delta\in Y$, equal to the $f$ image of the centers of the pancakes. The proof of the waist theorem uses the estimate of the volume of $t$-neighborhoods of $f^{-1}(y_\delta)$ for $t\gg \delta$, see \cite[Section 5]{mem2009}. This total estimate is assembled from the estimates for the $t$-neighborhoods of the center $c(P_i)$ in every pancake $P_i$, $\nu_t(c(P_i))\cap P_i$, looking like
\begin{equation}
\label{equation:pancake-estimate}
\frac{\vol \nu_t(c(P_i))\cap P_i}{\vol P_i} \ge \frac{\vol \nu_t(\mathbb S^k)}{\vol \mathbb S^n} + o(1)
\end{equation}
up to some error $o(1)$ tending to zero as $\delta\to+0$, with $k=1$ in our particular case. 

Now we want to check how our adjustment of the map $f$ affects this estimate on the size of the $t$-neighborhood of the set of centers of pancakes. We are satisfied when the center $c(P_i)$ is outside $\nu_\delta(x)$, in this case it corresponds to the domain where $f$ is originally defined before the adjustment. Thus we only have to bound from above the total volume of those $P_i$ whose centers turn out to be inside $\nu_\delta(x)$ therefore bounding the loss in the estimate when we exclude those ``bad'' pancakes.

If the total volume of those ``bad'' pancakes equals $V$ then by the $t$-neighborhood estimate for the centers of the pancakes we obtain that the $t$-neighborhood of the set of their centers, for $t\gg\delta$, has volume at least 
\[
V\frac{\vol \nu_t \mathbb S^1}{\vol \mathbb S^n} = \frac{V}{\vol \mathbb S^n} 2\pi v_{n-1} t^{n-1} (1 + o(1)),
\]
here we use the fact that the Minkowski content of $\mathbb S^1$ inside $\mathbb S^n$ actually equals its length, $o(1)$ means something tending to zero as $t\to+0$. At the same time the $t$-neighborhood of the centers of ``bad'' pancakes is contained in the $(t+\delta)$ neighborhood of the north pole, comparing the volumes we have
\[
\frac{V}{\vol \mathbb S^n} 2\pi v_{n-1} t^{n-1} (1 + o(1)) \le \vol \nu_{t+\delta}(x) = v_n (t+\delta)^n (1 + o(1)).
\]
For small $t$ and $\delta \ll t$, we see that $V = O(t)$ as $t\to +0$. This means that the estimate of the $t$-neighborhood after dropping ``bad'' pancakes is at least $1-O(t)$ times the original estimate, which is good for us, since in this theorem we are only interested in the limit case $t\to +0$.

After that we pass to the limit $N\to\infty$ (the number of pancakes), $\delta\to+0$ as it is done in \cite[Section 5]{mem2009}. In that process we will have to choose a single limit value $y_\delta\to y$ from a certain compactness argument. The values of $y_\delta$ are the images of the centers of the pancakes on the steps with positive $\delta$ tending to zero. Since every pancake is contained in the hemisphere around its center (see \cite[Section 4]{mem2009}), some of those centers must be from the southern hemisphere. Therefore the $y_\delta$ we choose are in the compact $f$-image of the southern hemisphere, not touched by our adjustment, and we can indeed choose an accumulation point of them in the going to the limit argument. 
\end{proof} 

\begin{remark}
In Theorem~\ref{theorem:disc-waist-k1} it is impossible to have two points of discontinuity on the sphere. Just because after dropping the north and the south poles it is possible to project the sphere onto its equator with all fibers of such a projection no longer than $\pi$. 

The proof above does not work in this situation because the identity map $\mathbb S^{n-1}\to \mathbb S^{n-1}$ cannot be extended to the northern hemisphere with boundary $\mathbb S^{n-1}$. But this observation shows that in Theorem \ref{theorem:disc-waist-k1} it is possible to drop arbitrarily many points from the sphere assuming the absence of such topological obstructions, in particular, assuming vanishing of the homotopy group $\pi_{n-1}(Y) = 0$.
\end{remark}

\subsection{Argument for the lower Minkowski content in Theorem~\ref{theorem:model-ball-waist}}
\label{ssec:minkowski-waist}

Appropriate versions of Propositions~\ref{proposition:cap-waist} and \ref{proposition:hyp-waist} would follow if we had an estimate similar to~\eqref{eq:k-measure-transform} generalized to the lower Minkowski content from the Riemannian $k$-volume, as well as a similar estimate for the conformal projections that we use in the proof. This all could be done if we understood the behavior of the (weighted) Minkowski content under smooth transformations, similar to the estimate made in~\cite[Theorem 5]{ak2016ball} for linear transformations. We could not find appropriate references and are aware of the lack of additivity of the Minkowski content that complicates matters. Therefore we follow another way and invoke the Gromov--Memarian proof directly rather then their result.

Note that we have made several maps that transformed the original $\mathbb S^n$ to 
what we previously called $C$ or $B$, the metric ball in the model space. Call it $B$ this time.

These maps behave smoothly except for one point (call it \emph{the north pole}), where it is not defined; but here we take Theorems~\ref{theorem:disc-waist} and \ref{theorem:disc-waist-k1} into account. We now actually work with the sphere without its north pole, carrying some Riemannian metric $g$ pulled back from $B$ under the maps constructed in Section \ref{section:model-riemannian-balls}.

We again look at the original proof in~\cite{grom2003,mem2009} and check what changes are needed there when we pass to a new Riemannian metric $g$. Let the standard metric of the sphere be $g_0$ and we compare it to the pull-back metric $g$. Since the map we consider is radially symmetric, at any point $x\in \mathbb S^n$ at distance $d$ from the south pole $s\in \mathbb S^n$ the metric $g$ compares to the metric $g_0$ as increased $\lambda^\parallel(d)$ times along the tangent to the geodesic $[s,r]$ and increased $\lambda^\perp(d)$ times in the perpendicular direction. 

Consider a standard geodesic subsphere $\mathbb S^k\subset \mathbb S^n$ passing through $s$. Remark~\ref{remark:k-ball-to-ball} just means that any $g_0$-geodesic $k$-dimensional ball $B^k(d)\subseteq \mathbb S^k$ centered at $s$ has the same Riemannian $k$-volume with respect to $g$ as it has with respect to $g_0$. From the radial symmetry any such $B^k(d)$ is also $g$-geodesic, in the stereographic image in $\mathbb R^n$ it is just a flat $k$-dimensional ball centered at the origin. In our terminology the equality of $k$-volumes means $\lambda^\parallel(d) \cdot \lambda^\perp(d)^{k-1} = 1$ for all $d>0$.
Moreover, Condition~\eqref{eq:condition on function f} in the current setting means that $\lambda^\perp(d) \ge \lambda^\parallel(d)$, this inequality held in $\mathbb R^n$ and is preserved under the inverse of the conformal stereographic projection.

Now we study the pancakes that were already mentioned in the proofs of Theorems \ref{theorem:disc-waist} and \ref{theorem:disc-waist-k1} for the particular case of $k=1$. The proof in \cite{grom2003,mem2009} decomposes the sphere into pancakes $P_1,\ldots, P_N$ ($N$ is a big power of two) so that every $P_i$ is $\delta$-close (for some $\delta \ll t$) to its respective $(n-k)$-dimensional geodesic subsphere $T_i$. 

For every pancake $P_i$, its volume was projected to $T_i$ and a measure $\mu_i$ in $T_i$ is obtained and shown to have certain concavity properties \cite[Section 4]{mem2009}. The point of maximum density of this measure is considered as a \emph{center} of $P_i$, $c_i\in T_i$. After that a Borsuk--Ulam-type argument \cite[Theorem 4]{mem2009} is used to adjust the partition into the $P_i$ so that the centers go to the same point under $f$, thus considering $\{c_1,\ldots, c_{2^\ell}\}$ as an approximation for the required fiber $f^{-1}(y)$. After that, fixing $t>0$, the estimate \eqref{equation:pancake-estimate} for the volume of the intersection of $B_{c_i}(t)\cap P_i$ is made and the sum of those estimates produces the total estimate for the volume of the $t$-neighborhood of the fiber as $\delta\to+0$ keeping $t$ fixed. 

In fact, in both \cite{grom2003} and \cite[Section 5]{mem2009} \emph{infinite partitions} were considered, where the pancakes $P_c$, after passing to the limit, were replaced by a continuum of Borel measures $\mu_c$ decomposing the volume measure of the sphere, with certain concavity properties ($\sin^{n-k}$-concave, \cite[Section 4]{mem2009}), and having convex support of dimension at most $(n-k)$ each. The argument in fact was a bit more complicated, since an introduction of a parameter $r>0$ was needed, which was used to smoothen the measures $\mu_c$ with radius $r$ and then make a continuous selection of the centers of maximal density of the smoothened measures. After that a going to the limit argument $r\to+0$ was applied, in our case we must go to the limit $t\to +0$ after that.

Now we check what changes in the pancake argument, if we take another metric $g$ instead of~$g_0$, simplifying our task by only estimating the lower Minkowski content and therefore considering the limit case $t\to +0$ (as mentioned above, $t$ goes to zero after $\delta\to +0$). Taking some fixed distance $t$ in the metric $g$ we see that its $t$-balls are now different in different parts of the sphere. Most of them for small $t$ will be approximate ellipsoids, those close to the north pole, where we have a discontinuity, will have size in the perpendicular direction tending to zero, since the metric $g$ tends to infinity there, along the perpendicular directions. Let us make a general observation that the Minkowski content will be estimated up to $(k+1)\alpha$ percent error, if we estimate the metric with $\alpha$ percent error for small $\alpha$, this allows us to describe the metric $g$ locally as a quadratic form at some point compared to another quadratic form $g_0$.

Now we look at a metric ball $B_c(t, g)$ centered at its respective pancake (call it $P_c$, since we are interested in its center and not interested in its number in the sequence) as in the original argument and want to understand how the intersection $B_c(t, g)\cap P_c$ looks like. Such a small ball $B_c(t,g)$ looks like a $g_0$-ball stretched in the south-north direction. Near the north pole $B_c(t,g)$ may be very much stretched, but outside a small neighborhood of the north pole they are only moderately stretched with bounded ratio of axes. The following observations produce the estimate we need: 

1) Outside a certain neighborhood of the north pole, at distance $d$ from the south pole $s$, for sufficiently small $t$, the ball $B_c(t,g)$ contains (up to a small percent error) the ball $B_c(t/\lambda^\perp(d), g_0)$ (here we use that $\lambda^\perp(d) \ge \lambda^\parallel(d)$). Then we may estimate the intersection $B_c(g,t)\cap P_c$ from below by the intersection $B_c(t/\lambda^\perp(d), g_0)\cap P_c$. For the latter we have an estimate from the original proof, but we need to multiply it by $\sqrt{\det g}/\sqrt{\det g_0}$ since we pass to the $g$-Riemannian volume from the $g_0$-Riemannian volume. 

2) We use the estimate \eqref{equation:pancake-estimate} for the intersection of the ball $B_c(t, g_0)$ with the pancake $P_c$, which is asymptotically $Ct^{n-k}$ for $t\to+0$. In our case we shrink the radius of the ball $1/\lambda^\perp(d)$ times, thus possibly getting the factor $\frac{1}{(\lambda^\perp(d))^{n-k}}$ in the original estimate, up to arbitrarily small percent of error if the ball is small (just because the original estimate sums to the volume of the $t$-neighborhood of $\mathbb S^k\subset\mathbb S^n$). After multiplying by $\sqrt{\det g}/\sqrt{\det g_0} = \lambda^\parallel(d) \lambda^\perp(d)^{n-1}$ to account for the change in the Riemannian density, we get the factor $\lambda^\parallel(d) \cdot \lambda^\perp(d)^{k-1} = 1$. This means that we have the estimate asymptotically equivalent to the original Gromov--Memarian estimate for $t\to +0$.

3) Near the north pole this does not work because there the balls $B_c(t,g)$ may become too thin in the perpendicular direction compared to $B_c(t, g_0)$, but this only happens in a very small neighborhood of the north pole that we may choose sufficiently small, only decreasing the result by an arbitrarily small percentage.

4) The resulting estimates obtained by the summation over all the pancakes will be tight for $t\to +0$ just because in the test case when $f^{-1}(y)$ is a geodesic $\mathbb S^k\subset \mathbb S^n$ passing through the south pole everything matches to give the Riemannian $k$-volume of $\mathbb S^k$ with respect to $g$, which we have observed to be the same as the Riemannian $k$-volume of $\mathbb S^k$ with respect to $g_0$.

\subsection{Another result on $t$-neighborhood waist in arbitrary norm}
\label{section:norm}

Let us produce another theorem on the waist of the ball, this time the ball will be the unit ball of a (possibly non-symmetric) norm in $\mathbb R^n$, and the $t$-neighborhood will be understood in the same norm. This is actually a version of \cite[Theorem 5.7]{klartag2016}, where $y$ is allowed to depend on $t$, improving the bound on the measure of the $t$-neighborhood in return:

\begin{theorem}
\label{theorem:waist-norm}
Suppose $K\subset\mathbb R^n$ is a convex body, $\mu$ is a finite log-concave measure supported in $K$, and $f : K\to Y$ is a continuous map to a $(n-k)$-manifold $Y$. Then for any $t \in [0, 1]$ there exists $y\in Y$ such that
\[
\mu (f^{-1}(y) + tK) \ge t^{n-k} \mu (K).
\] 
\end{theorem}

\begin{proof}
Again, the procedure is similar to the proof of Gromov's waist of the Gaussian measure theorem~\cite{grom2003}, see also a clear and detailed exposition in~\cite{klartag2016}.

\begin{itemize}
\item
Split $K$ into many \emph{$(n-k)$-pancakes} $P_1,\ldots, P_N$ of equal $\mu$-measures, that is sets $\delta$-close to $(n-k)$-dimensional affine subspaces $T_i$ of $\mathbb R^n$. We choose $\delta \ll t$, after that we will go to the limit $\delta\to +0$, as it was done in \cite[Pages 20--21]{klartag2016}.

\item
Once a pancake $P$ is close to an $(n-k)$-dimensional affine subspace $T$, choose the point $c(P)$ where the density of $\mu|_P$ projected to $T$ is maximal. The projected density will be log-concave by the Pr\'ekopa--Leindler inequality \cite{prekopa1971} and the set of points of maximum density will be convex.

\item
A certain version of the Borsuk--Ulam-type theorem~\cite[Section 2]{klartag2016} (actually the same as \cite[Theorem 4]{mem2009}) must be used when creating $N=2^\ell$ pancakes so that $f(c(P_1)) = f(c(P_2)) = \dots = f(c(P_N))$. In fact, in \cite{grom2003,mem2009,klartag2016} some effort was put on choosing the central point $c(P_i)$ continuously in $P_i$, this was achieved by allowing it to be an approximate central point, see \cite[Lemma 4.5]{klartag2016} for example. For our purposes we only need to know that it is possible to have the equality $f(c(P_1)) = f(c(P_2)) = \dots = f(c(P_N))$ for points $c(P_i)$ that are ``sufficiently central'', meaning that an estimate on the volume of $P_i$ intersected with a certain convex body centered at $c(P_i)$ is spoiled by an arbitrarily small percent of error compared to the real central point.

In fact, \emph{we will have to replace $f(c(P_1)) = f(c(P_2)) = \dots = f(c(P_N))$ this with another equality} after a closer investigation of the problem.

\item
Establish the following simple
\begin{lemma}
If $\mu$ is a log-concave measure in a convex body $L\subset \mathbb R^\ell$ with maximal density at $c\in L$ then
\[
\mu(t(L-c)+c) \ge  t^\ell\mu(L).
\]
\end{lemma}
\begin{proof}
Let $\rho$ be the density of $\mu$. To prove the lemma it is sufficient to note that, from the log-concavity, the density does not increase on moving along rays from $c$. Now we put $x = c + u s$, $u\in\mathbb S^{\ell-1}, s\ge 0$ and the required estimate 
\[
\int_{su \in t (L-c)} \rho(u s) \ell s^{\ell-1}\; ds du \ge t^\ell \int_{su \in (L-c)} \rho(u s) \ell s^{\ell-1}\; ds du
\]
follows from the easy inequality 
\[
\int_{s \le t S} \rho(u s) ds^\ell \ge t^\ell \int_{s \le S} \rho(u s) ds^\ell
\]
after the integration over the directions $u\in\mathbb S^{\ell-1}$.
\end{proof}

\item
We now apply the lemma to the approximating $(n-k)$-dimensional affine subspace $T_i$ of $P_i$ in place of $\mathbb R^\ell$, the projection of $\mu|_{P_i}$ to $T_i$ in place of the lemma's $\mu$, and $K\cap T_i$ in place of $L$. The conclusion of the lemma means that the convex body $t(K-c(P_i)) + c(P_i)$, which is the homothetic copy of $K$ with center of homothety $c(P_i)$ and ratio $t$, contains at least the fraction $t^{n-k}$ of $\mu|_{P_i}$ up to some error arising from replacing $\mu|_{P_i}$ by its projection to $T_i$. This relative error tends to $0$ while $\delta\to +0$ ($\delta$ is the pancakes' thickness), like in the original pancake arguments in \cite{mem2009,klartag2016}.

Note that compared to \cite[Lemma 4.5]{klartag2016} we use homotheties with variable centers, but a uniform estimate for approximate centers is possible if the center of homothety in $K$ is at some distance $d$ from $\partial K$. This can be obtained, for example, by making $\mu$ zero at the $d$-neighborhood of $\partial K$, obtaining some $f^{-1}(y_d)$ with a worse estimate, and then going to the limit $d\to +0$ so that $y_d$ also tends to some $y$ with the needed estimate in the limit.

\item
Note that $t(K-c(P_i)) + c(P_i) = tK + (1-t)c(P_i)$. But in the statement of the theorem we want to Minkowski-add $tK$ to the fiber of the map, hence the fiber has to be approximated by the points $(1-t)c(P_i)$ and the Borsuk--Ulam-type theorem in fact has to be applied to the equality
\[
f((1-t)c(P_1)) = f((1-t)c(P_2)) = \dots = f((1-t)c(P_N)).
\]
This equality is achievable by the same Borsuk--Ulam-type theorem, since the values $f((1-t)c(P_i))$ depend on $P_i$ continuously and this is all we need in the topological proof of the Borsuk--Ulam-type theorem. \emph{Note that the dependence of $y$ on $t$ arises here, for different $t$ we will have different sets of pancakes.}

\item
Eventually, summing the estimates in all the pancakes we arrive to an estimate
\[
\mu \bigcup_{t=1}^N (tK + (1-t)c(P_i)) \ge \mu \bigcup_{t=1}^N \left( P_i\cap (tK + (1-t)c(P_i))\right) \ge t^{n-k} \mu K - \epsilon(\delta)
\]
with $\epsilon(\delta)\to+0$ as $\delta\to+0$.

\item
Put $y_\delta = f((1-t)c(P_1)) = \dots = f((1-t)c(P_N))$ and make going to the limit as $\delta\to+0$ and $y_\delta$ depending on $\delta$ tends to some $y$. The details are skipped here and can be found in~\cite[Sections 4 and 5]{klartag2016}.

\end{itemize}

\end{proof}

\begin{remark}
The interchange of quantifiers in this theorem is really needed compared to Klartag's result. Consider the map $f : (-1,1)^n \to \mathbb R$ from the unit cube given by 
\[
f(x_1,\ldots,x_n) = \max\{x_1,\ldots, x_n\}.
\]
We want a fiber of this map whose every $t$-neighborhood (in the $\ell_\infty$ norm) has volume at least $t$ times the volume of the cube. Putting $t=1$ we must have the whole cube in the neighborhood, this forces us to choose the fiber $f^{-1}(0)$. But for small $t$ we have the volume of the neighborhood asymptotically $2nt + o(t)$, which is smaller than expected $t2^n$ from the statement of the theorem, for all $n\ge 3$. This shows that it is impossible to take the same fiber for all values of $t$.
\end{remark}

\begin{remark}
\label{remark:no-neighborhood-version}
The similar phenomenon happens for the Euclidean ball $B^n$. Take a point $p$ on its boundary and put
\[
f(x) = |x - p|.
\]
Consider the value $t=1$ in the above theorem, when both $p$ and $-p$ has to be covered by the $t$-neighborhood of the fiber. This forces to choose the fiber $f(x)=1$, the one passing through the origin. But for small $t$ the theorem promises the surface area at least $v_n/2$, while the surface area of the spherical cap
\[
|x-p|=1,\quad |x|\le 1
\]
can be roughly estimated by its projection to the direction orthogonal to $p$ as $\le 2(\sqrt{3}/2)^{n-1} v_{n-1}$ (the factor $2$ arises from the maximal slope of the projected hypersurface equal to $\pi/3$ with $\cos \pi/3 = 2$). But
\[
2\left(\frac{\sqrt{3}}{2}\right)^{n-1} v_{n-1} = 2(\sqrt{3}/2)^{n-1} \frac{\pi^{\frac{n-1}{2}}}{\Gamma\left(\frac{n+1}{2}\right)}\ \text{is exponentially smaller than}\ \frac{v_n}{2} = \frac{\pi^{\frac{n}{2}}}{2\Gamma\left(\frac{n+2}{2}\right)},
\]
for large values of $n$. In particular, this shows that it is problematic to state a tight for all $t$-neighborhoods version of the waist of the Euclidean ball theorem with the right order of quantifiers.
\end{remark}


\section{Monotonicity of the Minkowski content}
	\label{sec:monotonicity}

In order to complete the proof of Theorem~\ref{theorem:CAT-waist} using Lemma~\ref{lemma:cat-anti-lipschitz} we arrive at the question whether the Minkowski content is monotone with respect to $1$-Lipschitz maps of complete Riemannian manifolds. We are going to establish a certain weak monotonicity property and prove our estimates for the waist in terms of the upper Minkowski content.

\begin{theorem}
\label{theorem:mink-monotone}
Let $M$ be a complete Riemannian manifold and $N$ be another complete Riemannian manifold. If $X\subseteq M$ is compact and the map $f : X\to N$ is $1$-Lipschitz $(\dist(f(x), f(y))\le \dist(x,y))$ then
\[
\lmink_k(f(X), N) \le \umink_k (X, M).
\]
\end{theorem}

In case the lower and the upper Minkowski content for $X$ and $f(X)$ coincide, we just say that the Minkowski content decreases under a $1$-Lipschitz map. The proof of this theorem consists of essentially known results and a particular case of it, for the Euclidean space, is proved independently in~\cite[Proposition 4.1]{csi2017}.

\subsection{Constructions for the Minkowski content and the Kneser--Poulsen-type results}

We start with investigating the dependence of the Minkowski content on the ambient manifold. We will generalize the argument from~\cite{kneser1955,resman2013} to arbitrary Riemannian manifolds other than $\mathbb R^n$. Start with the case of comparing $\nu_t(X, M)$ and $\nu_t(X, M\times \mathbb R^\ell)$, where the product takes the product Riemannian structure. From the Fubini theorem we have
\[
\vol \nu_t(X, M\times \mathbb R^\ell) = \int_0^t \ell v_\ell u^{\ell-1} \vol \nu_{\sqrt{t^2-u^2}} (X, M)\; du = \int_0^t \vol \nu_s (X, M) \ell v_\ell (t^2 - s^2)^{(\ell-2)/2} s\; ds.
\]
The inequality $\vol \nu_s (X, M) \ge V v_{n-k} s^{n-k}$ for a certain constant $V$ for sufficiently small $s$ implies the following inequality for sufficiently small $t$:
\begin{multline*}
\vol \nu_t(X, M\times \mathbb R^\ell) \ge V \int_0^t v_{n-k} s^{n-k} \ell v_\ell (t^2 - s^2)^{(\ell-2)/2} s\; ds =\\
= V t^{n+\ell - k} v_{n-k} v_\ell \ell \int_0^1 x^{n-k+1} (1 - x^2)^{(\ell-2)/2} \; dx 
= V v_{n+\ell-k} t^{n+\ell-k}, 
\end{multline*}
where we use the equality:
\begin{equation*}
v_{n-k} v_\ell \ell \int_0^1 x^{n-k+1} (1 - x^2)^{(\ell-2)/2} \; dx = v_{n+\ell-k}.
\end{equation*}
This can be seen from the geometric interpretation of $v_{n+\ell-k}$ or directly from formula $v_{n} = \pi^{n/2}/\Gamma(n/2+1)$. 

We get the same Minkowski content estimate for the higher ambient dimension. The same applies to an estimate from above, thus establishing

\begin{lemma}
\[
\umink_k (X, M\times \mathbb R^\ell) \le \umink_k (X, M)\quad\text{and}\quad \lmink_k (X, M\times \mathbb R^\ell) \ge \lmink_k (X, M).
\]
\end{lemma}

Now we are going to extend this lemma to arbitrary emdeddings $M\subset M'$ of complete Riemannian manifolds:

\begin{lemma}
If $M\subset M'$ is an embedding of complete Riemannian manifolds and $M$ inherits its Riemannian structure from $M'$ then for a bounded subset $X\subset M$
\[
\umink_k (X, M') \le \umink_k (X, M)\quad\text{and}\quad \lmink_k (X, M') \ge \lmink_k (X, M).
\]
\end{lemma}

\begin{proof}
Let $\dim M' - \dim M = \ell$ and identify the normal bundle to $M$ in $M'$ with $M\times \mathbb R^\ell$ with its corresponding Riemannian structure. Then the exponential map 
\[
\exp : M\times \mathbb R^\ell \to M'
\]
has the property: For every $\epsilon > 0$ there exits $\delta$ such that the metric and the volumes are deformed at most $(1+\epsilon)$ times when we restrict $\exp$ to $M\times B_\delta(0)$ in a neighborhood of the bounded set $X$. Now from the definition of the Minkowski content of a subset of $M$, we see that those values may change $(1+\epsilon)^{m+\ell+1}$ times at most under the map $\exp$. Since $m+\ell$ is a constant here and $\epsilon$ is arbitrary, we obtain the result.
\end{proof}

\begin{proof}[Proof of Theorem~\ref{theorem:mink-monotone}]
Assume the contrary: $\lmink_k(f(X), N) > \umink_k (X, M)$. Moreover, slightly decreasing the metric on $N$ we may still assume the contrary and also assume that the map $f$ is $\lambda$-Lipschitz with constant $\lambda<1$. The general case is obtained by going to the limit.

Let us isometrically embed $M$ and $N$ to some big $\mathbb R^L$. Such an embedding is possible, see~\cite{nash1956}, and allows us to work in the Euclidean space still assuming $\lmink_k(f(X), \mathbb R^L) > \umink_k (X, \mathbb R^L)$. Increasing $L$ further if necessary, we may also assume $X\subset V$, $f(X)\subset V^\perp$ for two complementary orthogonal subspaces $V,V^\perp\in\mathbb R^L$.

Note that after the embedding the map need not remain $\lambda$-Lipschitz in the metric of $\mathbb R^L$, since the isometric Riemannian embedding is a local notion and is not necessarily isometric in the category of metric spaces. But we know that the distances decrease locally under $f$. Using the compactness of $X$ and we may find $d>0$ such that for every pair $x,y\in X$ with $|f(x) - f(y)|\le d$ we still have 
\[
|f(x) - f(y)| \le \lambda' |x-y|
\]
for some $\lambda' < 1$. After that in the definition of the Minkowski content we only consider $t < d/2$.

Let us produce a whole family of maps $f_\alpha : X\to \mathbb R^L$ for $\alpha\in[0,\pi/2]$ such that $f_0=\id$, $f_{\pi/2} = f$ and for any two points $x, y \in X$ with $|f(x)-f(y)|\le d$ the values 
\[
|f_\alpha(x) - f_\alpha(y)|
\]
are decreasing in $\alpha$. This map is constructed by putting
\[
f_\alpha(x) = \cos \alpha\; x \oplus \sin \alpha\; f(x)
\]
if we assume the decomposition $V\oplus V^\perp = \mathbb R^L$ with $X\subset V$ and $f(X)\subset V^\perp$.

Now we invoke the result of Csik\'os~\cite{csi1998}: If a finite family of balls in $\mathbb R^L$ moves so that the distance between any pair of their centers does not increase then the volume of the union of the family of balls does not increase. We also remark that in the proof of this result in~\cite{csi1998} it is only used that the distances decrease for those pairs of balls that intersect each other or start to intersect each other in the process of motion, because the derivative of the distance appears in the estimate only if there is a wall between Voronoi cells of the pair of balls inside their intersection. This assumption is satisfied in our setting.

Now a $t$-neighborhood of $X$ and $f(X)$ can be approximated by a union of a finite set of balls of radius $t$ having centers at $X$ and $f(X)$. Therefore, for a fixed $t$, the volumes $\vol \nu_t(f_\alpha(X), \mathbb R^L)$ must not increase in $\alpha\in[0,\pi/2]$. This makes a contradiction with our assumption and completes the proof.
\end{proof}

\subsection{Proof of Theorem \ref{theorem:CAT-waist} and Corollary \ref{corollary:cat1waist}}
\label{sec:cat-proofs}

To prove Theorem \ref{theorem:CAT-waist} we use Lemma~\ref{lemma:cat-anti-lipschitz} to obtain an anti-$1$-Lipschits map $h$ from $B'_q(R)\subset \mathbb M^n_\kappa$ to $B_p(R)\subset M$. Then apply Theorem~\ref{theorem:model-ball-waist} to the composition $f\circ h$ to have a set $X\subset \mathbb M^n_\kappa$ such that $f(h(X))$ is one point and $\lmink_k X \ge \vol_k B^k(R)$, where $B^k(R)$ is the $k$-dimensional ball in the model space. Put $Y = h(X)$, since the map $h$ in anti-$1$-Lipschitz, Theorem~\ref{theorem:mink-monotone} asserts that $\umink_k Y \ge \vol_k B^k(R)$, and we are done since $f(Y)$ is a single point.
Corollary \ref{corollary:cat1waist} follows immediately.

\begin{remark}
In principle the argument with pancakes from Section~\ref{ssec:minkowski-waist} could be applied to the sphere and the pull-back of the metric from $M$ on the sphere, this corresponds to replacing the metric $g$ in the above argument by another metric $g'\ge g$. The difficulty here is that the Gromov--Memarian estimate for $\mu_c B_c(t,g_0)$ is hard to extend from balls to small ellipsoids; in our proof of Theorem~\ref{theorem:model-ball-waist} the argument worked just because the ellipsoids were sufficiently symmetric and it turned out to be sufficient to estimate the ellipsoid by a ball put inside it.
\end{remark}

\subsection{Bezdek--Connelly monotonicity}

Continuing the study of the $1$-Lipschitz monotonicity property of the Minkowski content and using the result of~\cite{bezdek2002pushing} we can state a stronger result for the Minkowski content in the plane:

\begin{theorem}
\label{theorem:mink-monotoneR2}
If $X\subseteq \mathbb R^2$ is bounded and a map $f : X\to \mathbb R^2$ is $1$-Lipschitz then
\[
\lmink_k(f(X), \mathbb R^2) \le \lmink_k (X, \mathbb R^2)\quad\text{and}\quad \umink_k(f(X), \mathbb R^2) \le \umink_k (X, \mathbb R^2).
\]
\end{theorem}

\begin{proof}
In this case the positive solution of the Kneser--Poulsen conjecture in the plane for the monotonicity of the volume of the union of balls under $1$-Lipschitz maps of their centers can be invoked directly to show that the volume of a $t$-neighborhood is not increasing under $1$-Lipschitz maps. This establishes the result.
\end{proof}

A further generalization is possible if we look closer at the proof in~\cite{bezdek2002pushing}:

\begin{theorem}
\label{theorem:mink-monotoneRnplus2}
Assume $X\subseteq \mathbb R^n$ is compact and a map $f : X\to \mathbb R^n$ is $1$-Lipschitz and injective. Moreover, assume that after the natural inclusion $\mathbb R^n\subset\mathbb R^{n+2}$ the map $f$ can be joined by a homotopy $f_s$ with the identity, $f_0 = \id_X, f_1 = f$, so that $\dist(f_s(x), f_s(y))$ is non-increasing in $s$ for any $x,y\in X$. Then
\[
\lmink_k(f(X), \mathbb R^n) \le \lmink_k (X, \mathbb R^n)\quad\text{and}\quad \umink_k(f(X), \mathbb R^n) \le \umink_k (X, \mathbb R^n).
\]
\end{theorem}

\begin{proof}
Consider the $t$-neighborhood of $X$ in $\mathbb R^{n+2}$ and its boundary $\partial_t X$. While $t$ is fixed we may assume $X$ finite, the general case is obtained by going to the limit and approximating $X$ by finite sets. Consider the similarly defined $\partial_t f(X)$, the boundary of the $t$-neighborhood of $f(X)$ in $\mathbb R^{n+2}$. 

The sets $\partial_t X$ and $\partial_t f(X)$ are piece-wise smooth manifolds, built of patches of spheres. The metric projection of $\partial_t X$ to $\mathbb R^n$ projects it to the $t$-neighborhood $\nu_t X$ of $X$ in $\mathbb R^n$. The projection of the Riemannian $(n+1)$-dimensional measure on $\partial_t X$ is the uniform measure in $\nu_t X$ up to factor $2\pi/t$, see~\cite{archimedes225onsphere,bezdek2002pushing}. 

Consider now $\partial_t f_s(X)$ for varying $s$. This is always a boundary of a union of spheres and the pairwise distances between the centers do not increase with $s$. As it was shown in~\cite{bezdek2002pushing}, developing the Voronoi partition technique from~\cite{csi1998}, the $(n+1)$-volume of $\partial_t f_s(X)$ will then be non-increasing in $s$ thus establishing 
\[
\vol_n \nu_t f(X) \le \vol_n \nu_t X
\]
in $\mathbb R^n$ and completing the proof.
\end{proof}

\subsection{Gaussian version of the Minkowski content}

In order to improve the embedding and the $1$-Lipschitz monotonicity property of the Minkowski content we may just modify its definition. The general idea is to use the embedding to higher dimension right from the start and actually go to the limit when the added dimension goes to infinity. In this case the projection of a neighborhood of a set becomes a density proportional to the exponent of the (appropriately scaled) minus distance function squared.

\begin{definition}
For a compact set $X$ in a complete Riemannian manifold $M$ define the lower and upper \emph{Gaussian Minkowski content}
\[
\lgauss_k(X, M) = \liminf_{u\to +\infty} u^{\dim M - k} \int_M e^{-\pi u^2 \dist(x, X)^2} d\vol(x),
\]
\[
\ugauss_k(X, M) = \limsup_{u\to +\infty} u^{\dim M - k} \int_M e^{-\pi u^2 \dist(x, X)^2} d\vol(x),
\]
where $\vol$ is the Riemannian volume of $M$.
\end{definition}

For a tricky manifold $M$, that we do not encounter in this paper, this integral could diverge. This can be remedied by taking the integral not over the entire $M$, but over an $\varepsilon$-neighborhood of $X$ for some $\varepsilon > 0$. Choosing a different $\epsilon'$ only modifies the integral by an exponentially decaying (in $u$) term and therefore does not affect the limit in such adjusted definition.

Let $\dim M = n$ and write down the formula (Fubini's theorem for the subgraph of $y=e^{-\pi u^2 \dist(x, X)^2}$)
\[
\int_M e^{-\pi u^2 \dist(x, X)^2} d\vol(x) = \int_0^1  \vol \left\{ x : d(x, X)\le \sqrt{-\frac{\ln y}{\pi u^2}} \right\} dy
\]
and substitute $t = \sqrt{-\frac{\ln y}{\pi u^2}}$ to obtain
\begin{multline*}
\int_M e^{-\pi u^2 \dist(x, X)^2} d\vol(x) = \int_0^{+\infty} \vol \nu_t (X, M)\; d(-  e^{-\pi u^2t^2}) =\\
= \int_0^{+\infty} \frac{\vol \nu_t (X, M)}{v_{n-k} t^{n-k}}\;  v_{n-k} t^{n-k} d(-  e^{-\pi u^2t^2}).
\end{multline*}
These formulas actually express the right hand sides of the definitions of the Gaussian Minkowski content as averages (with density $v_{n-k} (ut)^{n-k} ( - e^{-\pi (ut)^2})'$ that is easily checked to integrate to $1$) of the right hand sides in the definition (e.g. \eqref{equation:lower-mink}) of the ordinary Minkowski content. In particular, we always have
\[
\lmink_k(X, M)\le \lgauss_k(X, M) \le \ugauss_k(X, M) \le \umink_k(X, M),
\]
and therefore there remain more chances of the equality 
\[
\lgauss_k(X, M) = \ugauss_k(X, M) = \gauss_k(X, M).
\]

\begin{theorem}
\label{theorem:gauss-product}
Assume $X\subseteq M$ and $Y\subseteq N$ are compacta in their respective complete Riemannian manifolds, and assume $\lgauss_m(Y, N) = \ugauss_m(Y, N) = \gauss_m(Y, N)$. Then
\[
\lgauss_{k+m} (X\times Y, M\times N) = \lgauss_k(X,M)\cdot \gauss_m(Y,N),
\]
\[
\ugauss_{k+m} (X\times Y, M\times N) = \ugauss_k(X,M)\cdot \gauss_m(Y,N).
\]
\end{theorem}

\begin{proof}
Observe that in the Riemannian product distance we have
\[
e^{-\pi u^2 \dist(x\times y, X\times Y)^2} = e^{-\pi u^2 \dist(x, X)^2}\cdot e^{-\pi u^2 \dist(y, Y)^2}.
\]
Applying the Fubini theorem gives the proof.
\end{proof}

\begin{theorem}
\label{theorem:gauss-embedding}
If $M\subset M'$ is an embedding of complete Riemannian manifolds and $M$ inherits its Riemannian structure from $M'$ then for a compact subset $X\subset M$
\[
\lgauss_k (X, M') = \lgauss_k (X, M),\quad \ugauss_k (X, M') = \ugauss_k (X, M)
\]
\end{theorem}

\begin{proof}
Observe that the case of the embedding $M\subset M\times \mathbb R^\ell$ follows from Theorem~\ref{theorem:gauss-product} by putting $Y=\{0\}\subset\mathbb R^\ell$. In the general case, for a given $\epsilon>0$, we may find $\delta>0$ such that the tubular $\delta$-neighborhood of $M$ in $M'$ (in the neighborhood of $X$) has Riemannian metric at most $\epsilon$ percent different from the metric of $M\times \mathbb R^\ell$ pushed forward to $M'$ by the exponential map of the tubular neighborhood. Since we may drop the $\delta$-distant part of the integral in the definition, it follows that we have the needed estimates up to $O(\epsilon)$ percent of error. This implies the result because $\epsilon$ is an arbitrary positive number.
\end{proof}

\begin{theorem}
\label{theorem:gauss-monotone}
Let $X\subseteq M$ be compact, $N$ be another complete Riemannian manifold, and the map $f : X\to N$ be $1$-Lipschitz $(\dist(f(x), f(y))\le \dist(x,y))$ then
\[
\lgauss_k(f(X), N) \le \lgauss_k (X, M),\quad \ugauss_k(f(X), N) \le \ugauss_k (X, M).
\]
\end{theorem}

\begin{proof}
The proof generally follows the lines of the proof of Theorem~\ref{theorem:mink-monotone}. We pass to embedding everything into $\mathbb R^L$ (this time the Gaussian Minkowski content is not affected at all) and make a homotopy $f_\alpha$ between the identity map and the final map $f$, which is $\lambda$-Lipschitz for pairs of points such that $|f(x) - f(y)|\le d$, with some $\lambda<1$.

In the definition of the Gaussian Minkowski content we now restrict the integration domain to $\dist(x, f_\alpha(X)) \le d/2$ in $\mathbb R^L$, this does not affect the result as was explained after the definition. Now we rewrite the Gaussian expression for $f_\alpha(X)$ as an average of the volumes of the $t$-neighborhoods of $f_\alpha(X)$ for $t\le d/2$. The result of Csik\'os~\cite{csi1998} then once again ensures that the volume of any $t$-neighborhood of $f_\alpha(X)$ is decreasing in $\alpha$, taking into account that the map is Lipschitz for distances $\le d$ and the radius of the neighborhood in question is at most $d/2$. Therefore the Gaussian expression does not increase in $\alpha$ as well.
\end{proof}

To conclude, we recall how the Gaussian Minkowski content is related to the waist theorems:

\begin{corollary}
Theorem \ref{theorem:CAT-waist} and Corollary \ref{corollary:cat1waist} remain valid if we replace $\umink$ by $\lgauss$.
\end{corollary}


\bibliography{../Bib/karasev,../Bib/akopyan}

\begin{thebibliography}{10}

\bibitem{ahk2016}
A.~Akopyan, A.~Hubard, and R.~Karasev.
\newblock Lower and upper bounds for the waists of different spaces.
\newblock 2016.
\newblock \href{https://arxiv.org/abs/1612.06926}{arXiv:1612.06926}.

\bibitem{ak2016ball}
A.~Akopyan and R.~Karasev.
\newblock A tight estimate for the waist of the ball.
\newblock {\em Bull. Lond. Math. Soc.}, 2017.
\newblock \href{http://arxiv.org/abs/1608.06279}{arXiv:1608.06279}.

\bibitem{almgren1965theory}
F.~J. Almgren.
\newblock The theory of varifolds: a variational calculus in the large for the
  $k$-dimensional area integrand.
\newblock 1965.

\bibitem{archimedes225onsphere}
{Archimedes of {S}yracuse}.
\newblock {\em {On Sphere and Cylinder}}.
\newblock ca. 225 BC.

\bibitem{bezdek2002pushing}
K.~Bezdek and R.~Connelly.
\newblock Pushing disks apart---the {K}neser-{P}oulsen conjecture in the plane.
\newblock {\em J. Reine Angew. Math.}, 553:221--236, 2002.

\bibitem{cheeger-gromov-taylor1982}
J.~Cheeger, M.~Gromov, and M.~Taylor.
\newblock Finite propagation speed, kernel estimates for functions of the
  laplace operator, and the geometry of complete riemannian manifolds.
\newblock {\em J. Differential Geom.}, 17:15--53, 1982.

\bibitem{csi1998}
B.~Csik\'os.
\newblock On the volume of the union of balls.
\newblock {\em Discrete Comput. Geom.}, 20(4):449--461, 1998.

\bibitem{csi2017}
B.~Csik\'os.
\newblock On the volume of the union of balls --- a review of the
  {K}neser--{P}oulsen conjecture.
\newblock In {\em New Trends in Intuitive Geometry}. Springer Verlag, 2017.

\bibitem{grom2001}
M.~Gromov.
\newblock {CAT}(k)-spaces: construction and concentration.
\newblock {\em Zap. Nauchn. Sem. S.-Peterburg. Otdel. Mat. Inst. Steklov. (POMI)}, 280:101--140, 2001.
\newblock
  \href{http://www.ihes.fr/~gromov/PDF/6[110].pdf}{www.ihes.fr/~gromov/PDF/6[110].pdf}.

\bibitem{grom2003}
M.~Gromov.
\newblock Isoperimetry of waists and concentration of maps.
\newblock {\em Geom. Funct. Anal.}, 13:178--215, 2003.

\bibitem{grom2010}
M.~Gromov.
\newblock Singularities, expanders and topology of maps. part 2: from
  combinatorics to topology via algebraic isoperimetry.
\newblock {\em Geom. Funct. Anal.}, 20(2):416--526, 2010.

\bibitem{gromov2014number}
M.~Gromov.
\newblock Number of questions.
\newblock {\em preprint}, 2014.

\bibitem{karvol2013}
R.~Karasev and A.~Volovikov.
\newblock Waist of the sphere for maps to manifolds.
\newblock {\em Topology Appl.}, 160(13):1592--1602, 2013.
\newblock \href{http://arxiv.org/abs/1102.0647}{arXiv:1102.0647}.

\bibitem{klartag2016}
B.~Klartag.
\newblock Convex geometry and waist inequalities.
\newblock {\em Geom. Funct. Anal.}, 27(1):130--164, 2017.
\newblock \href{http://arxiv.org/abs/1608.04121}{arXiv:1608.04121}.

\bibitem{kneser1955}
M.~Kneser.
\newblock Einige bemerkungen \"uber das minkowskische fl\"achenmass.
\newblock {\em Arch. Math. (Basel)}, 6:382--390, 1955.

\bibitem{mem2009}
Y.~Memarian.
\newblock On {G}romov's waist of the sphere theorem.
\newblock {\em J. Topol. Anal.}, 03(01):7--36, 2011.
\newblock \href{http://arxiv.org/abs/0911.3972}{arXiv:0911.3972}.

\bibitem{nash1956}
J.~Nash.
\newblock The imbedding problem for {R}iemannian manifolds.
\newblock {\em Ann. of Math.}, 63(1):20--63, 1956.

\bibitem{prekopa1971}
A.~Pr\'ekopa.
\newblock Logarithmic concave measures with application to stochastic
  programming.
\newblock {\em Acta Sci. Math. (Szeged)}, 32:301--316, 1971.

\bibitem{resman2013}
M.~Resman.
\newblock Invariance of the normalized {M}inkowski content with respect to the
  ambient space.
\newblock {\em Chaos Solitons Fractals}, 57:123--128, 2013.

\bibitem{toponogov2006differential}
V.~A. Toponogov.
\newblock {\em Differential geometry of curves and surfaces}.
\newblock Birkh\"auser Boston, Inc., Boston, MA, 2006.
\newblock A concise guide, With the editorial assistance of Vladimir Y.
  Rovenski.

\bibitem{vaaler1979}
J.~D. Vaaler.
\newblock A geometric inequality with applications to linear forms.
\newblock {\em Pacific J. Math.}, 83(2):543--553, 1979.

\end{thebibliography}

\bibliographystyle{abbrv}
\end{document}